\documentclass[12pt]{article}

\usepackage{latexsym,hyperref,amssymb,amsfonts,amsmath,amsthm,nccmath}
\usepackage[margin=1.5cm]{geometry}
\usepackage[dvipsnames]{xcolor}

\newcommand{\comment}[1]{}

\numberwithin{equation}{section}
\newtheorem{theorem}{Theorem}

\theoremstyle{definition}

\let\oldproofname=\proofname
\renewcommand{\proofname}{\rm\bf{\oldproofname}}

\def \Z {{\mathbb{Z}}}
\def \P {\mathcal{P}}

\def \S {\mathcal{S}}
\def \P {\mathcal{P}}
\def \Y {\mathcal{Y}}
\def \B {\mathcal{B}}
\def \A {\mathcal{A}}
\def \V {\mathcal{V}}
\def \C {\mathcal{C}}

\def \ord {{\rm ord}}
\def \md#1{{\:({\rm mod}\ #1)}}

\title{\bf Steiner triple systems without parallel classes}

\author{\hspace{-0.7cm}
\begin{minipage}[c]{8cm}
\begin{center}
Darryn Bryant\\
The University of Queensland \\
Department of Mathematics \\
Qld 4072, Australia\\
{\tt db@maths.uq.edu.au}
\end{center}
\end{minipage}
\begin{minipage}[c]{8cm}
\begin{center}
Daniel Horsley\\
School of Mathematical Sciences\\
Monash University\\
Vic 3800, Australia\\
{\tt danhorsley@gmail.com}
\end{center}
\end{minipage}
}

\date{ }

\begin{document}
\maketitle\thispagestyle{empty}
\def\baselinestretch{1.5}\small\normalsize

\begin{abstract}
We construct Steiner triple systems without parallel classes for an infinite number of orders
congruent to $3 \md 6$. The only previously known examples have order $15$ or $21$.
\vspace{0.25cm}

\end{abstract}

A {\em Steiner triple system} of order $v$ is a pair $(V,\B)$ where $V$ is a $v$-set of {\em
points} and $\B$ is a collection of $3$-subsets of $V$, called {\em triples}, such that each
(unordered) pair of points occurs in exactly one triple. Steiner triple systems are fundamental in
design theory. Kirkman \cite{Kir} proved in 1847 that there exists a Steiner triple system of order
$v$ if and only if $v\equiv 1$ or $3\md 6$ (see \cite{ColRos}).

A {\em parallel class} in a Steiner triple system $(V,\B)$ is a subset of $\B$ that partitions $V$.
Of course, only Steiner triple systems of order congruent to $3 \md 6$ may contain parallel
classes. It is a long-standing conjecture, dating back to at least 1984, that systems without
parallel classes exist for all orders congruent to $3 \md 6$ except $3$ and $9$ (see
\cite{ColRos}). In spite of this, the only previously known examples without parallel classes have
order $15$ or $21$. The unique Steiner triple system of order $9$ has a parallel class, $10$ of the
$80$ nonisomorphic Steiner triple systems of order $15$ have no parallel classes (see
\cite{ColRos}), and $12$ of the $1772$ nonisomorphic $4$-rotational Steiner triple systems of order
$21$ have no parallel classes \cite{MatRos}.

In this paper we construct Steiner triple systems without parallel classes for an infinite number
of orders congruent to $3 \md 6$.
Our construction is similar to one used by Schreiber \cite{Sch} and Wilson \cite{Wil}.
Problems concerning parallel classes in Steiner triple systems
are well studied. Ray-Chaudhuri and Wilson \cite{RayWil} famously proved that for every order
congruent to $3\md 6$ there exists a Steiner triple system whose block set can be partitioned into
parallel classes. As it turned out, this problem had been solved earlier by Lu and his solution was eventually published in \cite{Lu}. The study of parallel classes in Steiner triple systems also relates to problems concerning matchings in certain hypergraphs, and Alon et al. \cite{AloKimSpe} give a lower bound on the number of disjoint triples in any Steiner triple system. Also see \cite{PipSpe} and the surveys \cite{ColRos} and \cite{RosCol}.

The existence of {\em almost parallel classes} (sets of disjoint triples containing all the points
except one) in Steiner triple systems of order congruent to $1 \md 6$ has also been investigated.
It is conjectured that for all $v\equiv 1\md 6$ except $v=13$ there exists a Steiner triple system
of order $v$ with no almost parallel class \cite{RosCol}. Two sparse infinite families of such
Steiner triple systems have been found, one by Wilson (see \cite{RosCol}) and one by the authors
\cite{BryHorNoAPC}. Also see \cite{Mes} for a recent related result.

For a prime $p$, let $\ord_p(x)$ denote the multiplicative order of $x$ in $\Z_p$. Let $\P$ be the
set of odd primes given by $p\in\P$ if and only if $\ord_{p}(-2)\equiv 0\md 4$. Observe that $\P$
contains all of the infinitely many primes congruent to $5\md 8$, no primes congruent to $3$ or
$7\md 8$, and some primes congruent to $1\md 8$. Clearly, if $p\equiv 3\md 4$ is prime, then
$\ord_p(-2) \not\equiv 0 \md 4$. If $p\equiv 5\md 8$ is prime, then $-2$ is a quadratic non-residue
in $\Z_p$ by the law of quadratic reciprocity and it follows that $\ord_p(-2) \equiv 0 \md 4$. Our
theorem states that there is a Steiner triple system of order $v$ with no parallel class whenever
$v\equiv 27\md{30}$ and $p\in\P$ for each prime divisor $p$ of $v-2$. Thus, for any list
$p_1,\ldots,p_t$ of (not necessarily distinct) primes from $\P$ in which the number of primes
congruent to $5\md 6$ is odd, we obtain a Steiner triple system of order $v=5p_1\cdots p_t+2$ with
no parallel class.

\begin{theorem}\label{mainthm}
Let $\V$ be the set of positive integers given by $v\in\V$ if and only if $v \equiv 27 \md{30}$ and $\ord_{p}(-2)\equiv
0\md 4$ for each prime divisor $p$ of $v-2$. Then $\V$ is infinite and for each $v\in\V$ there
exists a Steiner triple system of order $v$ having no parallel class.
\end{theorem}

\begin{proof}
It is clear from the remarks preceding the theorem that $\V$ is infinite. We now show that if $v\in\V$, then  there exists a Steiner triple system of order $v$ having no parallel class.
Let $v \in\V$, define $n$ by $v=5n+2$ (note that $n$ is an odd integer because $v \equiv 27 \md{30}$), and let  $n=p_1p_2\cdots p_t$ be the prime factorisation of $n$. By our hypotheses $\ord_{p_i}(-2)\equiv
0\md 4$ for $i=1,2,\ldots,t$.
Let $p_0=5$ and let $H=\Z_{p_1}\times\Z_{p_2}\times\cdots\times\Z_{p_t}$ and
$G=\Z_{p_0}\times H$ be additive groups with identities
$0_H=(0,0,\ldots,0)$ and $0_G=(0,0_H)$.
If $m=\frac ab$ where $a$ and $b$ are integers such that $\gcd(b,p_i)=1$ for $i=0,\ldots,t$, and $g=(g_0,g_1,g_2,\ldots,g_t)\in G$, then we define $mg$ by
$mg=(ab^{-1}g_0,ab^{-1}g_1,\ldots,ab^{-1}g_t)$ where $ab^{-1}g_i$ is evaluated in the field $\Z_{p_i}$ for $i=0,\ldots,t$.

Let
$$\A^0=\{\{g,g',g''\}:g,g',g''\in G,|\{g,g',g''\}|=3,g+g'+g''=0\}.$$
Observe that each pair of distinct elements of $G$ occurs in at most one triple of $\A^0$, and that the pairs that occur in no triple of $\A^0$ are exactly the edges of the graph $X$ with vertex set $V(X)=G$ and edge set
$$E(X)=\{\{g,-2g\}:g\in G\setminus\{0_G\}\}.$$
Thus, $0_G$ is an isolated vertex in $X$ and each vertex $g\in G\setminus\{0_G\}$ has exactly $2$
neighbours in $X$, namely $-\frac 12 g$ and $-2g$. That is, $X$ is a collection of disjoint cycles
together with an isolated vertex. To each edge of $X$ assign a weight equal to the sum in $G$ of
its endpoints. It follows that for each $g\in G\setminus\{0_G\}$ there is a unique edge in $X$
having weight $g$ (namely the edge $\{-g,2g\}$). Moreover, the edge of $X$ having weight $g$ has an
adjacent edge of weight $-2g$. Henceforth we denote the edges of $X$ by their weights and describe
the cycles of $X$ by listing their edges in cyclically ordered tuples. So the cycle containing an
arbitrary edge $g$ of $X$ has vertices $-g,2g,\ldots,\textstyle\frac 12g$, edges
$\{-g,2g\},\{2g,-4g\},\ldots,\{\textstyle\frac 12g,-g\}$, and is denoted by
$$(g,-2g,\ldots,-\tfrac12g).$$

Let $\theta: H\setminus\{0_H\} \rightarrow H\setminus\{0_H\}$ be given by
$\theta(h)=-2h$, and let $\mathcal{Y}$ be the set of orbits of $\theta$.
Since $\ord_{p_i}(-2)\equiv 0\md 4$ for $i=1,2,\ldots,t$, we have $|Y|\equiv 0\md 4$ for each $Y\in\Y$.
Thus, each cycle of $X$ except
$$C=\left((1,0_H),(3,0_H),(4,0_H),(2,0_H)\right)$$
is of the form
$$\left( (a,h),(3a,\theta(h)),(4a,\theta^2(h)),\ldots,(2a,\theta^{|Y|-1}(h))\right) $$
for some $a\in\Z_5$, $h\in Y$ and $Y\in\Y$. Note that $E(X)\setminus E(C)=\Z_5\times(H\setminus\{0_H\})$.

We claim that there is a function $\gamma:\Z_5\times(H\setminus\{0_H\})\rightarrow\{1,2\}$ such that for all $g,g'\in\Z_5\times(H\setminus\{0_H\})$
\begin{enumerate}
\item [{\rm (1)}] $\gamma(g)\neq\gamma(-2g)$; and
\item [{\rm (2)}] if $g+g'\in\Z_5\times\{0_H\}$, then $\gamma(g)=\gamma(g')$.
\end{enumerate}
For each $Y \in \mathcal{Y}$, let $-Y=\{-h:h \in Y\}$. It is clear that if $Y \in \mathcal{Y}$,
then $-Y \in \mathcal{Y}$. Thus, either $-Y=Y$ or $Y$ and $-Y$ are distinct orbits in
$\mathcal{Y}$. For each $Y \in \mathcal{Y}$, choose a representative element $h_Y$ ensuring that if
$Y$ and $-Y$ are distinct, then $h_{-Y}=-h_Y$. Write $H \setminus \{0_H\}$ as a union of disjoint
sets $H_1$ and $H_2$ given by
$$
H_1=\displaystyle\bigcup_{Y\in\mathcal Y}\{\theta^i(h_Y):i=0,2,\ldots,|Y|-2\}\quad{\text{and}}\quad
H_2=\displaystyle\bigcup_{Y\in\mathcal Y}\{\theta^i(h_Y):i=1,3,\ldots,|Y|-1\}.
$$
Since $|Y|$ is even for each $Y\in\Y$, $H_1$ and $H_2$ are well defined. For each $a\in\Z_5$ and
each $h\in H\setminus\{0_H\}$ define the function
$\gamma:\Z_5\times(H\setminus\{0_H\})\rightarrow\{1,2\}$ by $\gamma((a,h))=1$ if $h\in H_1$ and
$\gamma((a,h))=2$ if $h\in H_2$. It is clear that $\gamma$ satisfies (1). It also satisfies (2)
because $|Y|\equiv 0\md 4$ for each $Y\in\Y$, and because we have $h_{-Y}=-h_Y$ when $Y$ and $-Y$
are distinct.

We now construct a Steiner triple system of order $v=|G|+2$ which we will prove has no parallel class.
Let $\infty_1$ and $\infty_2$ be two new points (not in $G$) and let
$$\B^\infty=\{\{-g,2g,\infty_{\gamma(g)}\}:g\in\Z_5\times(H\setminus\{0_H\})\}.$$
Also, let $((\Z_5 \times \{0_H\}) \cup \{\infty_1,\infty_2\},\B^*)$ be any Steiner triple system
of order $7$ such that $\{0_G,\infty_1,\infty_2\}\notin\B^*$ and let $\B^0=\A^0 \setminus \{\{0_G,(1,0_H),(4,0_H)\},\{0_G,(2,0_H),(3,0_H)\}\}$. It is easily seen that
$$\S=(G\cup\{\infty_1,\infty_2\},\B^0\cup\B^\infty\cup\B^*)$$
is a Steiner triple system of order $v$. We now show that $\S$ has no parallel class.

For a contradiction, suppose $\C$ is a parallel class of $\S$. Define the weight of each triple of
$\S$ to be the sum in $G$ of its points, where the points $\infty_1$ and $\infty_2$ are treated as
$0_G$. Since the sum of the elements of $G$ is $0_G$, the sum of the weights of the triples in $\C$
is also $0_G$. Now note the following properties concerning weights of triples which we will use in
the arguments that follow. Every triple in $\B^0$ has weight $0_G$, every triple in $\B^*$ has weight in $\Z_5\times\{0_H\}$, and every triple in $\B^\infty$ has weight in
$\Z_5\times(H\setminus\{0_H\})$. The weight of the triple that contains both $\infty_1$ and
$\infty_2$ is not $0_G$.

If the triple that contains both $\infty_1$ and $\infty_2$ is in $\C$, then no other triple of
$\B^*$ is in $\C$ (because any two triples of a Steiner triple system of order $7$ intersect) and
no triple of $\B^\infty$ is in $\C$ (because every triple of $\B^\infty$ contains either $\infty_1$
or $\infty_2$). This means that the remaining triples of $\C$ are from $\B^0$. But this is a
contradiction because every triple in $\B^0$ has weight $0_G$, and the weight of the triple that
contains both $\infty_1$ and $\infty_2$ is not $0_G$. We conclude that $\C$ contains distinct
triples $T_1$ and $T_2$ such that $\infty_1\in T_1$ and $\infty_2\in T_2$, and that the remaining
triples of $\C$ are all from $\B^0\cup\B^*$.

Since every triple in $\B^0\cup\B^*$ has weight in $\Z_5\times\{0_H\}$, the sum of the weights of
$T_1$ and $T_2$ is also in $\Z_5\times\{0_H\}$. This means that either both $T_1$ and $T_2$ are in
$\B^*$ or both are in $\B^\infty$ (because every triple in $\B^*$ has weight in
$\Z_5\times\{0_H\}$, and every triple in $\B^\infty$ has weight in
$\Z_5\times(H\setminus\{0_H\})$). The triples $T_1$ and $T_2$ cannot both be in $\B^*$ because any
two triples of $\B^*$ intersect. On the other hand, if $T_1$ and $T_2$ are both in $\B^\infty$,
then the fact that the sum of the weights of $T_1$ and $T_2$ is in $\Z_5\times\{0_H\}$ contradicts
Property (2) of $\gamma$. We conclude that $\S$ has no parallel class.
\end{proof}

\vspace{0.3cm} \noindent{\bf Acknowledgements:}\quad The authors acknowledge the support of the
Australian Research Council via grants DE120100040, DP120100790, DP120103067, DP150100530 and DP150100506.

\end{document}